\documentclass[11pt]{amsart}
\usepackage{amscd}
\usepackage{amssymb}
\usepackage[matrix,arrow]{xy}
\newtheorem{theorem}{Theorem}[section]

\newtheorem{proposition}[theorem]{Proposition}

\newtheorem{lemma}[theorem]{Lemma}

\newtheorem{corollary}[theorem]{Corollary}

\newtheorem{remark}[theorem]{Remark}

\newtheorem{definition}[theorem]{Definition}

\newcommand{\CCC}{\mathbb{C}}

\newcommand{\QQQ}{\mathbb{Q}}
\newcommand{\ZZZ}{\mathbb{Z}}
\newcommand{\OOO}{\mathcal{O}}

\newcommand{\Tor}{\mathrm{Tor}}
\newcommand{\Aut}{\mathrm{Aut}}
\newcommand{\Vect}{\mathrm{Vect}}

\begin{document}
\title[The C*-algebra of a vector bundle]{The C*-algebra of a vector bundle}
\thanks{The author was partially supported by
NSF grant \#DMS--0801173}
\author{Marius Dadarlat}
\address{Department of Mathematics, Purdue University, West
Lafayette IN 47907, U.S.A.} \email{mdd@math.purdue.edu}
\date{\today}

\begin{abstract} We prove that the Cuntz-Pimsner algebra $O_E$ of a vector bundle $E$ of rank $\geq 2$ over
a compact metrizable space $X$ is  determined up to an isomorphism of $C(X)$-algebras by the  ideal $(1-[E])K^0(X)$ of the K-theory ring $K^0(X)$. Moreover, if $E$ and $F$ are vector bundles of rank $\geq 2$, then
a unital embedding of $C(X)$-algebras $O_E\subset O_F$ exists if and only if
$1-[E]$ is divisible by $1-[F]$ in the ring $K^0(X)$.
We  introduce related, but more computable K-theory and
cohomology invariants for $O_E$ and study their completeness.
   As an application we classify
  the unital separable continuous fields with fibers isomorphic to the Cuntz algebra $O_{m+1}$ over a finite connected CW complex
 $X$ of dimension $d\leq 2m+3$ provided that the cohomology of $X$ has no $m$-torsion.
\end{abstract}
\maketitle
\section{Introduction}
 Let $E\in \Vect(X)$ be a locally trivial complex vector bundle
 over a compact  Hausdorff space  $X$. If we endow $E$ with a hermitian metric, then the space $\Gamma(E)$ of all continuous sections  of $E$
becomes a finitely generated projective Hilbert $C(X)$-module, whose isomorphism class does not depend on the choice of the  metric.
 Since the action of $C(X)$ is central,  $\Gamma(E)$ is naturally a Hilbert $C(X)$-bimodule.
Let $O_E$ denote the Cuntz-Pimsner algebra associated to $\Gamma(E)$ as defined in \cite{Pim:CK-alg}.
Since $\Gamma(E)$ is projective, $O_E$ is isomorphic to the Doplicher-Roberts algebra of $\Gamma(E)$, see \cite{DPZ}. Let us recall that if $\mathcal{E}$ is the
Hilbert $C(X)$-module $\oplus_{n\geq 0} \Gamma(E)^{\otimes n}$, then
$O_E$ is obtained as the quotient of the Toeplitz (or tensor) C*-algebra
$T_E$ generated by the multiplication operators $T_\xi:\mathcal{E}\to \mathcal{E}$,
$T_\xi(\eta)=\xi\otimes \eta$, $\xi\in \Gamma(E)$, $\eta\in \mathcal{E}$, by the ideal of ``compact operators" $K(\mathcal{E})$.
If $X$ is a point, then $E\cong\CCC^{n}$ for some $n\geq 1$, and $O_E$ is isomorphic to the Cuntz
algebra $O_{n}$, with the convention that $O_1=C(\mathbb{T})$.
 In  the general case,  $O_E$ is a  locally trivial unital $C(X)$-algebra (continuous field)
whose fiber at $x$ is isomorphic to the Cuntz algebra $O_{n(x)}$, where $n(x)$
is the rank of the fiber $E_x$ of $E$, see \cite[Prop.~2]{Vasselli:1}.

The motivation for this paper comes from an informal question of  Cuntz:
What are the invariants of $E$ captured by the $C(X)$-algebra $O_E$?
In other words, how are $E$ and $F$ related if there is a $C(X)$-linear
$*$-isomorphism $O_E\cong O_F$.
We have shown in \cite{Dad:bundles-fdspaces} that if $X$ has finite covering dimension, then all separable unital $C(X)$-algebras with fibers isomorphic to a fixed Cuntz algebra $O_n$, $n\geq 2$, are automatically locally trivial.
Thus it is also natural to ask which of these algebras are isomorphic to Cuntz-Pimsner algebras associated to a vector bundle of constant
rank $n$.

If $E$ is a line bundle, then $O_E$ is commutative with spectrum homeomorphic to  the circle bundle of $E$, see
\cite{Vasselli:2}. One verifies that if $E,F\in Vect_1(X)$ and $X$ is path-connected, then $O_E\cong O_F$ as $C(X)$-algebras if and only if $E\cong F$ or $E\cong \bar{F},$ where $\bar{F}$ is the conjugate
of $F$, see Proposition~\ref{line-bundles}.
 In view of  this property, we shall only consider vector bundles of rank $\geq 2$.
 In the first part of the paper we answer the isomorphism question for $O_E$.
\begin{theorem}\label{Thm:complete-obstruction_O_E}
Let $X$ be a compact metrizable space and let $E,F \in \Vect(X)$ be complex vector bundles  of rank $\geq 2$. Then $O_E$ embeds as a unital $C(X)$-subalgebra of $O_F$ if and only if  there is $h\in K^0(X)$ such that $1-[E]=(1-[F])h$.
Moreover, $O_E\cong O_F$ as $C(X)$-algebras if and only if there is $h$ as above of virtual rank one.
\end{theorem}
Thus  the principal ideal  $(1-[E])K^0(X)$
determines $O_E$ up to isomorphism and an inclusion of principal ideals   $(1-[E])K^0(X)\subset  (1-[F])K^0(X)$
 corresponds to unital embeddings $O_E\subset O_F$.  In particular if $E\in \Vect_{m+1}(X)$,
 then $O_E\cong C(X)\otimes O_{m+1}$ if and only if $[E]-1$ is divisible by $m\geq 1$.

 Let $\widetilde K^0(X)=\ker(K^0(X)\stackrel{rank}{\longrightarrow}H^0(X,\ZZZ))$ be the subgroup of $K^0(X)$ corresponding to elements
of virtual rank zero,  and set $[\widetilde E]:=[E]-\mathrm{rank}(E)\in \widetilde K^0(X)$.
We denote by $H^*(X,\ZZZ)$ the \v{C}ech cohomology.
Using the nilpotency of $\widetilde K^0(X)$ we derive the following:
\begin{theorem}\label{thm:main_result-intro2} Let $X$ be a compact metrizable space of finite dimension $n$. Suppose that  $\Tor(K^0(X),\ZZZ/m)=0$. If
$E,F\in Vect_{m+1}(X)$, then $O_E\cong O_F$ as $C(X)$-algebras if and only if
$([\widetilde E]-[\widetilde F])\left(\sum_{k=1}^n (-1)^{k-1}m^{n-k}[\widetilde F]^{k-1}\right)$ is divisible by
$m^n$ in $\widetilde K^0(X)$.
 \end{theorem}
 In view of Theorem~\ref{Thm:complete-obstruction_O_E} it is natural to seek explicit and computable invariants (e.g. characteristic classes) of a vector bundle $E$ that depend only on the
principal ideal  $(1-[E])K^0(X)$ and hence which
are invariants of $O_E$.

For each $m\geq 1$, consider the sequence of polynomials $p_n\in \ZZZ[x]$,
\begin{equation}\label{def:pn}
 p_n(x)= \ell(n)\,m^n\,\log\left(1+\frac{x}{m}\right)_{[n]}=  \sum_{k=1}^n (-1)^{k-1} \frac{\ell(n)}{k}m^{n-k}x^k,
\end{equation}
where $\ell(n)$ denotes the least common multiple of the numbers $\{1,2,...,n\}$ and the index
$[n]$ indicates that the formal series of the natural logarithm is truncated after its  $n$th term.
\begin{theorem}\label{thm:k-theory-inv}
 Let $X$ be a  finite CW complex of dimension $d$ and let $E,F \in \Vect_{m+1}(X)$.
If $O_E\cong O_F$  as $C(X)$-algebras,  then
$p_{\lfloor d/2 \rfloor}([\widetilde E])-p_{\lfloor d/2 \rfloor}([\widetilde F])$ is divisible by $m^{\lfloor d/2 \rfloor}$ in $\widetilde K^0(X)$.
\end{theorem}
For $x\in \mathbb{R}$, we set
 $\lfloor x \rfloor :=\max\{k\in \ZZZ: k\leq x\}$ and $\lceil x\rceil:=\min\{k\in \ZZZ: k\geq x\}$. Theorem~\ref{thm:k-theory-inv}
extends to finite dimensional compact metrizable spaces: if $n\geq 1$ is an
integer such that  $\widetilde K^0(X)^{n+1}=\{0\}$,
then $p_{n}([\widetilde E])-p_{n}([\widetilde F])$ is divisible by $m^{n}$ in $\widetilde K^0(X)$ whenever $O_E\cong O_F$ as $C(X)$-algebras. The same conclusion holds for infinite dimensional spaces
$X$ but in that case  $n$ depends on $E$ and $F$.

Concerning the completeness of the above invariant we have the following:
\begin{theorem}\label{thm:main_result2-intro} Let $X$ be a  finite  CW complex of dimension $d$.
Suppose that $m$ and $\lfloor d/2 \rfloor!$ are relatively prime and that $\Tor(H^*(X,\ZZZ),\ZZZ/m)=0$. If $E,F\in Vect_{m+1}(X)$,
then  $O_E\cong O_F$ as $C(X)$-algebras if and only if  $p_{\lfloor d/2 \rfloor}([\widetilde  E])-p_{\lfloor d/2 \rfloor}([\widetilde F])$ is divisible by $m^{\lfloor d/2 \rfloor}$ in $\widetilde K^0(X)$.
\end{theorem}
The condition that  $m$ and $\lfloor d/2 \rfloor!$ are relatively prime is necessary.
To show this, we take $m=2$ and let $X$ be
the complex projective space  $\CCC \mathrm{P}^2$.  Then $K^0(X)$ is isomorphic to the polynomial ring $\ZZZ[x]$
with $x^3=0$, \cite{Kar:k-theory}. Let $E$ and $F$ be bundles with K-theory classes  $[E]=3+3x$ and $[F]=3+x$.
Then $ [\widetilde E]=3x$ and $[\widetilde F]=x$, so that $p_2(3x)-p_2(x)=8(x-x^2)$  is divisible by $4$ and yet
 Theorem~\ref{thm:main_result-intro2} shows that $O_E\ncong O_F$,  since $([\widetilde E]-[\widetilde F])(2-[\widetilde F])=4x-2x^2$ is not divisible
by $4$. The vanishing of $m$-torsion is also necessary in both Theorems~\ref{thm:k-theory-inv} and~\ref{thm:main_result2-intro} as it is seen by taking $m=2$ and $X=\mathbb{R}\mathrm{P}^2\vee \CCC \mathrm{P}^2$, where $\mathbb{R}\mathrm{P}^2$ is the real projective space. Indeed, let
$E,F\in \mathrm{Vect}_3(X)$ be such that $F$ is trivial and
$[\widetilde E]|_{\mathbb{R}\mathrm{P}^2}=z$
is the generator of $\widetilde K^0(\mathbb{R}\mathrm{P}^2)=\ZZZ/2$ and
$[\widetilde E]|_{\mathbb{C}\mathrm{P}^2}=2x+2x^2$. Then
$([\widetilde E]-[\widetilde F])(2-[\widetilde F])=(z+2x+2x^2)(2)=4x+4x^2$ is divisible by $4$ and yet
$O_E\ncong C(X)\otimes O_3$ by Theorem~\ref{Thm:complete-obstruction_O_E} since $[E]-1=2+z+2x+2x^2$ is not divisible by $2$.

 Next we exhibit characteristic classes of $E$ which are invariants of $O_E$. For each $n\geq 1$ consider the  polynomial $q_n\in \ZZZ[x_1,...,x_n]$:
\begin{equation}\label{qn}
 q_n=\sum_{ k_1+2k_2+\dots+nk_n=n}(-1)^{k_1+\dots+k_n-1} m^{n-(k_1+\cdots+k_n)}\frac{n!\,(k_1+\dots+k_n-1)!}{1!^{k_1}\cdots{n!}^{k_n}\,k_1!\cdots k_n!}\,x_1^{k_1}\cdots x_n^{k_n}.
 \end{equation}
Thus $q_1(x_1)=x_1$, $q_2(x_1,x_2)=mx_2-x_1^2$, $q_3(x_1,x_2,x_3)=m^2x_3-3mx_1x_2+2x_1^3$, etc.
Let $ch_n$ be the integral characteristic classes that appear in the Chern character,
$ch=\sum_{n\geq 0} \frac{ 1}{n!}ch_n$.
\begin{theorem}\label{thm:cohomology-inv}
   Let $X$ be a  compact metrizable space  and let $E,F \in \Vect_{m+1}(X)$.
If $O_E\cong O_F$  as $C(X)$-algebras,  then
$q_n(ch_1(E),...,ch_n(E))-q_n(ch_1(F),...,ch_n(F))$ is divisible by $m^{n}$ in $H^{2n}(X,\ZZZ)$, for all $n\geq 1$.
 \end{theorem}
Reducing mod $m^n$ it follows that the sequence ${q}_n(\dot{ch}_1(E),...,\dot{ch}_n(E))\in H^{2n}(X,\ZZZ/m^n)$, $n\geq 1$, is an invariant of  the $C(X)$-algebra $O_E$.

Let us denote by $\OOO_{m+1}(X)$ the set of isomorphism classes
of unital separable $C(X)$-algebras with all fibers isomorphic to $O_{m+1}$.  In the second part of the paper we study the range of the map
$\Vect_{m+1}(X)\to \OOO_{m+1}(X)$.
 This relies on the computation
of the homotopy groups of $\Aut(O_{m+1})$ of \cite{Dad:bundles-fdspaces}. If $T$ is a set, we denote by $|T|$ its cardinality.

\begin{theorem}\label{thm:main_result-intro1} Let $X$ be a finite CW complex  of dimension $d$. Suppose that  $m\geq \lceil (d-3)/2\rceil$ and  $\Tor(H^*(X,\ZZZ),\ZZZ/m)=0$. Then
each element of $\OOO_{m+1}(X)$
  is isomorphic to  $O_E$ for some $E$ in $Vect_{m+1}(X)$. Moreover $|\OOO_{m+1}(X)|=|\widetilde K^0(X)\otimes \ZZZ/m|=|\widetilde H^{even}(X,\ZZZ/m)|$.
\end{theorem}
The hypotheses of Theorem~\ref{thm:main_result-intro1} are necessary. Indeed, to see that the condition
$m\geq \lceil(d-3)/2\rceil$ is necessary even in the absence of torsion, we note that $Vect_3(S^8)=\{*\}$ since $\pi_7(U(3))=0$ by
\cite{Kervaire:homotopy-unitary}, whereas $\OOO_{3}(S^8)\cong \pi_7(\Aut(O_3))\cong\ZZZ/2$ by \cite{Dad:bundles-fdspaces}.
To see that the condition on torsion is necessary  when $m\geq \lceil(d-3)/2\rceil$,
we note that if $X=\mathbb{R}\mathrm{P}^2$, then $Vect_3(SX)=\{*\}$ since $\widetilde K^0(SX)\cong K^1(X)=\{0\}$ and $\dim(SX)=3$, whereas $\OOO_{3}(SX)\cong
K^1(X,\ZZZ/2)\cong \ZZZ/2$ by \cite{Dad:bundles-fdspaces}.

 The study of the map $Vect_{m+1}(X)\to \OOO_{m+1}(X)$ simplifies considerably if $X$ is a suspension
 as explained in Theorem~\ref{thm:susp} from Section~\ref{section:susp}.

  In Section~\ref{section:classification} we prove Theorems~\ref{Thm:complete-obstruction_O_E} -~\ref{thm:main_result-intro2}.
Theorem~\ref{thm:k-theory-inv} is proved in Section~\ref{section:2} and
Theorem~\ref{thm:main_result2-intro} is proved in Section~\ref{section:2algebra}.
The proofs of Theorem~\ref{thm:cohomology-inv} and Theorems~\ref{thm:main_result-intro1} are given in Section~\ref{section:2+} and respectively Section~\ref{section:3}.

Cuntz-Pimsner algebras come with a natural $\mathbb{T}$-action and hence with a $\ZZZ$-grading. The  question studied by Vasselli in \cite{Vasselli:1}
of when $O_E$ and $O_F$ are isomorphic as $\ZZZ$-graded  $C(X)$-algebras is not directly related to the questions
 addressed in this paper. I would like to thank Ezio Vasselli for making me aware of the isomorphism
$O_E\cong O_{\bar{E}}$ for line bundles, see \cite{Vasselli:2}.
 \section{When is $O_E$ isomorphic to $O_F$?}\label{section:classification}
 In this section we prove Theorems~\ref{Thm:complete-obstruction_O_E}-\ref{thm:main_result-intro2} and  discuss the case of line bundles.

 \begin{proof} (of Theorem~\ref{Thm:complete-obstruction_O_E}) We identify $K_0(C(X))$ with $K^0(X)$. Let $\iota_E$ denote the canonical unital inclusion    $C(X)\to O_E$.
 By \cite{Pim:CK-alg}, the K-theory group $K_0(O_E)$ fits into an exact sequence
\[
\xymatrix{
K_0(C(X))\ar[r]^{1-[E]}&K_0(C(X))\ar[r]^-{(\iota_E)_*}&K_0(O_E),
}
\]
where $1-[E]$  corresponds to the multiplication map  by the element $1-[E]$.
Therefore $\ker (\iota_E)_*=(1-[E])K^0(X)$.
Suppose that $\phi:O_E\to O_F$ is a $C(X)$-linear unital $*$-homomorphism.
Then $\phi \circ \iota_E=\iota_F$ and hence $\ker (\iota_E)_*\subset \ker (\iota_F)_*$.
It follows that  $(1-[E])K^0(X)\subset (1-[F])K^0(X)$ and hence
$1-[E]=(1-[F])h$ for some $h\in K^0(X)$. If $\phi$ is an
isomorphism, we deduce similarly that $1-[F]=(1-[E])k$ for some $k\in K^0(X)$. In that case
 $rank(E_x)=rank(F_x)$ for each $x\in X$ and $h$ must have constant virtual rank
equal to one.

Conversely, suppose that there is $h\in K^0(X)$  such that $(1-[E])=(1-[F])h$.
We have $O_E \otimes O_\infty \cong O_E$ and  $O_F \otimes O_\infty \cong O_F$  by \cite{BK}.
We are going to show the existence of a unital $C(X)$-linear embedding
$O_E\subset O_F$ by producing an element
$\chi\in KK_X(O_E,O_F)$
which maps $[1_{O_E}]$ to $[1_{O_F}]$ and then appeal to \cite{Kir:Michael}.
If the virtual rank of $h$
is equal to one and hence $h$ is invertible in the ring $K^0(X)$, we show that
 $\chi$ is a $KK_X$-equivalence and that will imply
that $O_E$ is isomorphic to
$O_F$.
Since the operation of suspension is an isomorphism in $KK_X$, it suffices to show that
there is $\eta\in KK_X(SO_E,SO_F)$, respectively $\eta\in KK_X(SO_E,SO_F)^{-1}$,
such that $\eta \circ [S\iota_E]=[S\iota_F]$.

Let us recall that the mapping cone
of a $*$-homomorphism $\alpha:A\to B$
is
\[C_{\alpha}=\{(f,a)\in C([0,1],B)\oplus A: f(0)=\alpha(a), f(1)=0\}. \]
If $\alpha$ is a morphism of continuous $C(X)$-algebras, then
the natural extension
\[
\xymatrix{
0\ar[r] & SB\ar[r]^{\lambda} \ar[r]& C_\alpha
\ar[r]^{p}& A\ar[r]&0,
}
\]
where $\lambda(f)=(f,0)$ and $p(f,a)=a$, is an extension of continuous $C(X)$-algebras.

Let $\mathrm{KK}(X)$ denote the additive category with objects  separable $C(X)$-algebras and
morphisms from $A$ to $B$ given by the group $KK_X(A,B)$.
Nest and Meyer have shown that $\mathrm{KK}(X)$ is a triangulated category \cite{Meyer-Nest:BC}.
In particular, for any diagram of separable $C(X)$-algebras and $C(X)$-linear $*$-homomorphisms
\[
\xymatrix{
SB \ar[r]^{\lambda}\ar[d]_{S\varphi} & C_{\alpha}\ar@{-->}[d]^{\gamma} \ar[r]^{p} & A \ar[r]^{\alpha}\ar[d]_{\psi} & B\ar[d]_{\varphi}
\\
SB' \ar[r]^{\lambda'} & C_{\alpha'} \ar[r]^{p} & A' \ar[r]^{\alpha'} & B'
}
\]
such that the right square commutes in $\mathrm{KK}(X)$,  there is $\gamma\in KK_X(C_\alpha,C_{\alpha'})$ that makes to remaining squares commute. If the right square commutes up to homotopy of $C(X)$-linear $*$-homomorphisms, then
$\gamma$ can be chosen to be a  $C(X)$-linear $*$-homomorphism, see \cite[Prop. 2.9]{Sc:axiom-homology}. The general case is proved in a similar way, see \cite[Appendix A]{Meyer-Nest:BC}.
Let us note that if $\varphi$ and $\psi$ are $KK_X$-equivalences, so is $\gamma$ by the exactness of the Puppe sequence and the five-lemma.

We need another general observation.
If
\[
\xymatrix{
0\ar[r] & J\ar[r]^{j} \ar[r]& B
\ar[r]^{\pi}& B/J\ar[r]&0
}
\]
is an extension of separable  continuous $C(X)$-algebras, then there is a surjective $C(X)$-linear
$*$-homomorphism $\mu:C_j\to SB/J$, $\mu(f,b)=\pi\circ f$, and hence
 an extension
of separable continuous $C(X)$-algebras
\[
\xymatrix{
0\ar[r] & CJ \ar[r]& C_j
\ar[r]^-{\mu}& SB/J\ar[r]&0.
}
\]
If $B/J$ is nuclear then it is $C(X)$-nuclear and since $CJ$ is $KK_X$-contractible it follows that $\mu$ must be a $KK_X$-equivalence, see \cite{Bauval:KKX}.

Let us recall that $O_E$ is defined by the extension
\(
\xymatrix{
0\ar[r] &\mathcal{K}(\mathcal{E})\ar[r]^{j_E}&T_E\ar[r]^{\pi_E}\ar[r]&0.
}
\)

 By Theorem 4.4 and Lemma 4.7 of \cite{Pim:CK-alg}
there is a commutative diagram  in $\mathrm{KK}(X)$:
\[
\xymatrix{
\mathcal{K}(\mathcal{E})\ar[d]_{[\mathcal{E}]}\ar[r]^{[j_E]}&T_E
\\
C(X)\ar[r]^{[id]-[E]} &C(X)\ar[u]_{[i_E]}
}
\]
where both vertical maps are $KK_X$-equivalences. Here $i_E$ is the canonical unital inclusion, $[\mathcal{E}]$ is the class in $KK_X(\mathcal{K}(\mathcal{E}),C(X))$ defined by the  bimodule $ \mathcal{E}$ that implements the strong Morita equivalence between $\mathcal{K}(\mathcal{E})$ and $C(X)$ and $[E]$ is the class in $KK_X(C(X),C(X))$ defined by the finitely generated  $C(X)$-module $\Gamma(E)$. Note that
\([id]-[E]\in KK_X(C(X),C(X))\) induces the multiplication map $K^0(X)\stackrel{1-[E]}{\longrightarrow} K^0(X)$.
Pimsner's statements refer to ordinary $KK$-theory but his constructions and arguments are natural
and  preserve the $C(X)$-structure.

After tensoring the  C*-algebras in the above diagram by $O_\infty$ we can realize
$[\mathcal{E}]^{-1}$ and $[id]-[E]$ as $KK_X$ classes of $C(X)$-linear  $*$-homomorphisms $\psi_E$ and respectively $\alpha_E$.
This is easily seen by using the identification $KK_X(C(X),B)\cong KK(\CCC,B)$ for $B$ a $C(X)$-algebra and noting that
 $\mathcal{K}(\mathcal{E})$ contains a full projection since $\mathcal{E}$ is isomorphic
to $C(X)\otimes \ell^2(\mathbb{N})$ by Kuiper's theorem. Thus we obtain a diagram
\[
\xymatrix{
\mathcal{K}(\mathcal{E})\otimes O_\infty \ar[r]^{j_E\otimes id}&T_E\otimes O_\infty
\\
C(X)\otimes O_\infty\ar[u]^{\psi_E}\ar[r]^-{\alpha_E} &C(X)\otimes O_\infty\ar[u]_{i_E\otimes id}
}
\]
that commutes in $\mathrm{KK}(X)$.
We  construct a similar  diagram for the bundle $F$.
Let $H:C(X)\otimes O_\infty \to C(X)\otimes O_\infty$ be a  $C(X)$-linear $*$-homomorphism
which sends $[1_{C(X)\otimes O_\infty}]$ to $h\in K^0(X)\cong K_0(C(X)\otimes O_\infty)$. Note that $H$ is a $KK_X$-equivalence whenever $h$
is invertible in the ring $K^0(X)$.
Since $1-[E]=(1-[F])h$ by assumption,
 the  diagram
\[
\xymatrix{
C(X)\otimes O_\infty\ar[d]_{H}\ar[r]^-{\alpha_E} &C(X)\otimes O_\infty
\\
C(X)\otimes O_\infty\ar[r]^-{\alpha_F} &C(X)\otimes O_\infty\ar@{=}[u]
}
\]
commutes in the category $\mathrm{KK}(X)$.
The  proof of the theorem is based on the following commutative diagram in $\mathrm{KK}(X)$:
\[
\xymatrix{
&SO_E\otimes O_\infty &&&\\
ST_E\otimes O_\infty \ar[r]^{\lambda_E}& C_{j_E}\otimes O_\infty\ar[u]^{\mu_E\otimes id}\ar[r] &\mathcal{K}(\mathcal{E})\otimes O_\infty\ar[r]^{j_E\otimes id}&T_E\otimes O_\infty
\\
SC(X)\otimes O_\infty \ar[u]^{Si_E\otimes id}\ar@{=}[d]\ar[r]& C_{\alpha_E}\ar[r]\ar@{-->}[d]^{\gamma}
\ar@{-->}[u]_{\gamma_E}
&C(X)\otimes O_\infty\ar[d]^{H}\ar[u]_{\psi_E} \ar[r]^{\alpha_E}&C(X)\otimes O_\infty \ar@{=}[d]\ar[u]_{i_E\otimes id}
\\
SC(X)\otimes O_\infty\ar[d]_{Si_F\otimes id} \ar[r]& C_{\alpha_F}\ar[r] \ar@{-->}[d]^{\gamma_F}&C(X)\otimes O_\infty \ar[d]^{\psi_F}\ar[r]^{\alpha_F}&C(X)\otimes O_\infty \ar[d]^{i_F\otimes id}
\\
ST_F\otimes O_\infty \ar[r]^{\lambda_F}& C_{j_F}\otimes O_\infty\ar[d]_{\mu_F\otimes id}\ar[r] &\mathcal{K}(\mathcal{F})\otimes O_\infty\ar[r]^{j_F\otimes id}&T_F\otimes O_\infty\\
&SO_F\otimes O_\infty &&&
}
\]
The elements $\gamma$, $\gamma_E$ and
$\gamma_F$ are constructed  as explained above since the category $\mathrm{KK}(X)$ is triangulated.
Moreover  $\gamma_E$ and
$\gamma_F$ are $KK_X$-equivalences since they are induced by the $KK_X$-equivalences
$i_E, \psi_E$ and $i_F,\psi_F$. Arguing similarly, we see that $\gamma$  is a $KK_X$-equivalence whenever
$h$ is invertible in the ring $K^0(X)$.
The morphisms $\mu_E$ and $\mu_F$ are associated to the Toeplitz extensions
$0\to \mathcal{K}(\mathcal{E})\to T_E\to O_E\to 0$ and respectively  $0\to \mathcal{K}(\mathcal{F})\to T_F\to O_F\to 0$ and they are also $KK_X$-equivalences as we argued earlier in the proof.
Let us note that $(\mu_E\otimes id)\circ \lambda_E\circ (Si_E\otimes id_{O_\infty})=S\iota_E\otimes id_{O_\infty}$
and $(\mu_F\otimes id)\circ \lambda_F\circ (Si_F\otimes id_{O_\infty})=S\iota_F\otimes id_{O_\infty}$ by a simple direct verification.
In this way we are able to find an element
 \[\eta:=[\mu_F\otimes id]\circ \gamma_F\circ \gamma \circ \gamma_E^{-1}\circ [\mu_E\otimes id]^{-1}\in KK_X(SO_E\otimes O_\infty,SO_F\otimes O_\infty)\cong KK_X(SO_E,SO_F)\]
  such that the following diagram
commutes in $\mathrm{KK}(X)$.

\[
\xymatrix{
{}&{SC(X)}\ar[dr]^{S\iota_F}\ar[dl]_{S\iota_E}&{}\\
{SO_E}\ar[rr]^{\eta}&{}&{SO_F}\\
}
\]
Unsuspending, we find
$\chi\in KK_X(O_E,O_F)$
which maps $[1_{O_E}]$ to $[1_{O_F}]$, since $\chi\circ [\iota_E]=[\iota_F]$.
If $h$ is invertible in $K^0(X)$, then $\chi$ is a $KK_X$-equivalence. It follows that
 $O_E\cong O_F$ as $C(X)$-algebras by applying
 Kirchberg's   isomorphism theorem \cite{Kir:Michael}.
 If $\chi$ is just a $KK_X$-element which preserves the classes of the units, we invoke again \cite{Kir:Michael} in order to lift
 $\chi$ to a unital $C(X)$-linear embedding $O_E\subset O_F$.
\end{proof}
Set $T_n(b):=\sum_{k=1}^{n} (-1)^{k-1}\,m^{n-k}\,b^{k-1}$. Then $(m+b)T_n(b)=m^n-(-b)^n$.
\begin{lemma}\label{lemma:B} Let $R$ be a  commutative  ring such that $R^{n+1}=\{0\}$ for some $n\geq 1$ and let $a,b\in R$.  If
there is $h\in R$ such that $a=b+mh+bh$, then
$(a-b)T_{n}(b)=m^nh$. Conversely, if  $(a-b)T_{n}(b)=m^nh$ for some $h\in R$, then $m^n(a-b-mh-bh)=0$.
\end{lemma}
\begin{proof} Suppose that $a-b=(m+b)h$ for some $h\in R$.
Then $(a-b)T_{n}(b)=(m^n-(-b)^n)h=m^nh-(-b)^nh$. Since $(-b)^nh\in R^{n+1}$ must vanish it follows that
$(a-b)T_{n}(b)=m^nh$.
 Conversely, suppose that $(a-b)T_{n}(b)=m^n h$ for some $h\in R$. Then
 $(a-b)T_{n}(b)(m+b)=m^n(m+b) h$ and hence $(a-b)(m^n-(-b)^n)=m^n(m+b) h$. But $(a-b)(-b)^n=0$ since
 $R^{n+1}=\{0\}$. Therefore $m^n(a-b-mh-bh)=0$.
 \end{proof}

We are now prepared to prove Theorem~\ref{thm:main_result-intro2}.

\begin{proof}
Since $\dim(X)\leq n$, we can embed $X$ in $\mathbb{R}^{2n+1}$ and then find a decreasing sequence $X_i$ of
polyhedra  whose intersection is $X$.
We have ${\widetilde K}^0(X_i)^{n+1}=\{0\}$ since $\dim(X_i)\leq 2n+1$ (see the next section for further discussion). It follows that
${\widetilde K}^0(X)^{n+1}=\{0\}$ since  $\widetilde K^0(X)\cong \varinjlim \widetilde K^0(X_i)$.
Let us write $[E]-1=m+a$ and $[F]-1=m+b$ where $a=[\widetilde E]$ and
$b=[\widetilde F]\in \widetilde K^0(X)$.
By  Theorem~\ref{Thm:complete-obstruction_O_E}, $O_E\cong O_F$ if and only if
 $[E]-1=([F]-1)(1+h)$ for some $h\in \widetilde K^0(X)$ and hence if and only if
 $a=b+mb+bh$ for some $h\in \widetilde K^0(X)$. With this observation we conclude the proof by applying Lemma~\ref{lemma:B}.
 \end{proof}

For a hermitian bundle $E$ we denote by $\bar{E}$ the conjugate bundle, by $E_0$ the set of all nonzero elements in $E$ and by $S(E)$
the unit sphere bundle of $E$. 

\begin{proposition}\label{line-bundles} Let $E$ and $F$ be hermitian line bundles over a path-connected compact  metrizable space $X$.
Then $O_E\cong O_F$ as $C(X)$-algebras if and only if  either $E\cong F$ or $E\cong \bar{F}$.
\end{proposition}
\begin{proof} Vasselli has shown that for a line bundle $E$,  $O_E\cong C(S(E))$ as $C(X)$-algebras \cite{Vasselli:2}. Therefore it suffices to show that
 there is a homeomorphism of sphere-bundles $\phi:S(E)\to S(F)$ if and only if  either $E\cong F$ or $E\cong \bar{F}$. The isomorphism $O_E\cong O_{\bar{E}}$ was noted in \cite{Vasselli:2}. One can argue as follows.
 The conjugate bundle $\bar{E}$ has the same underlying real  vector bundle as $E$ but with opposite complex structure; the identity map $E\to \bar{E}$ is conjugate linear. If we endow $\bar{E}$ with the conjugate hermitian metric
 it follows that the identity map is fiberwise norm-preserving and hence it identifies $S(E)$ with $S(\bar{E})$.
 
Conversely,
suppose that there is a homeomorphism of sphere-bundles $\phi:S(E)\to S(F)$. By homogeneity we
can extend $\phi$ to a fiber-preserving homeomorphism $\Phi:E \to F$ such that $\Phi(E_0)\subset F_0$.
Let $p_E:E\to X$ be the projection map and let $i_E$ be the inclusion map $(E,\emptyset)\subset (E,E_0)$.
Let us recall  that the underlying real vector bundle $E_{\mathbb{R}}$ has a canonical preferred orientation 
which yields a Thom class $u_E\in H^2(E,E_0,\ZZZ)$, see \cite[ch.9, ch.14]{Milnor-Stasheff}. Since $X$ is path connected, 
$\ZZZ\cong H^0(X,\ZZZ)\cong H^2(E,E_0,\ZZZ)=\ZZZ u_E$ by the Thom isomorphism.
The first Chern class  $c_1(E)$ is equal to the Euler class  $e(E_{\mathbb{R}})=(p_E^*)^{-1}i_E^*(u_E)$.
The map $\Phi$ induces a commutative diagram
 \[
\xymatrix{
\ZZZ u_E=H^2(E,E_0,\ZZZ)\ar[r]^-{i_E^*}&H^2(E,\ZZZ)&H^2(X,\ZZZ)\ar[l]_{p_E^*}\\
\ZZZ u_F=H^2(F,F_0,\ZZZ)\ar[u]^{\Phi^*}\ar[r]^-{i_F^*}&H^2(F,\ZZZ)\ar[u]^{\Phi^*}&H^2(X,\ZZZ)\ar[l]_{p_F^*}\ar@{=}[u]
}
\]
 Since $\Phi$ is a homeomorphism, $\Phi^*(u_F)=\pm u_E$
and hence $c_1(E)=\pm c_1(F)$. It follows that either $E\cong F$ or $E\cong \bar{F}$.
\end{proof}

\section{K-theory invariants of $O_E$}\label{section:2}
In this section we construct a sequence $(\mu_n(E))_{n}$ of K-theory invariants of $O_E$.
The class of the trivial  bundle of rank $r$ is denoted by $r\in K^0(X)$.
All the elements of the ring $\widetilde K^0(X)$ are nilpotent \cite{Kar:k-theory}.

Recall that for $E\in\Vect_{m+1}(X)$ we denote by $[\widetilde E]$  the
$\widetilde K^0(X)$-component
$[E]-(m+1)$ of $[E]$.
We introduce the following equivalence relation on $\widetilde K^0(X)$:
 $a\sim b$, if and only
if $a=b+mh+bh$ for some $h\in \widetilde K^0(X)$. Rewriting $a=b+mh+bh$ as $m+a=(m+b)(1+h)$
in $K^0(X)$, it becomes obvious that $\sim$ is an equivalence relation since $1+h$ in invertible in the ring $K^0(X)$ with inverse $1+\sum_{k\geq 1}\,(-1)^k h^k$. Moreover, if $E,F\in\Vect_{m+1}(X)$, then $(1-[E])K^0(X)=(1-[F])K^0(X)$ if and only if
$[\widetilde E]\sim [\widetilde F]$.

Note that with our new notation, Theorem~\ref{Thm:complete-obstruction_O_E} shows that $O_E\cong O_F$ $\Rightarrow$
$[\widetilde E]\sim [\widetilde F]$. In other words the equivalence class of $[\widetilde E]$
in $\widetilde K^0(X)/\sim$ is an invariant of $O_E$. In order to obtain more computable invariants,
for each $m\geq 1$, we use  the sequence of polynomials $p_n\in\ZZZ[x]$ introduced in \eqref{def:pn}. It is immediate that
 \begin{equation}\label{eqn:poly_n}
  p_{n+1}(x)=\frac{\ell(n+1)}{\ell(n)} m p_{n}(x)+(-1)^n \frac{\ell(n+1)}{n+1}x^{n+1}.
 \end{equation}

The first five polynomials in the sequence are:
\begin{itemize}
\item[] $p_1(x)=x,$
\item[] $p_2(x)=2mx-x^2,$
\item[] $p_3(x)=6m^2x-3mx^2+2x^3,$
\item[] $p_4(x)=12 m^3x -6 m^2x^2  + 4 mx^3  - 3x^4,$
\item[] $p_5(x)=60m^4x  - 30m^3x^2   + 20m^2x^3  - 15mx^4  + 12 x^5$.
\end{itemize}
\begin{lemma}\label{lemma:ABC} For any $n\geq 1$ there are polynomials $u_n,s_{n+1}\in \ZZZ[x,y]$ and
$v_n\in \ZZZ[x]$ such that each monomial of $s_{n+1}$ has total degree $\geq n+1$ and
 \begin{itemize}
 \item[(i)] $p_n(x+y)=p_n(x)+p_n(y)+xy \,u_n(x,y)$,
  \item[(ii)]  $p_n(x+my+xy)=p_n(x)+m^n\, v_n(y)+s_{n+1}(x,y).$
 \end{itemize}
\end{lemma}
\begin{proof}
It follows from the binomial formula that for any polynomial $p\in \ZZZ[x]$ with $p(0)=0$ there is
a polynomial $u\in \ZZZ[x,y]$ such that  $p(x+y)=p(x)+p(y)+xy \,u(x,y)$. This proves (i).
Let us now prove (ii). Set $V_n(x)=\sum_{k=1}^n (-1)^{k-1} x^k/k=\log(1+x)_{[n]}.$
 Then \(p_n(x)=\ell(n)m^n V_n(x/m)\). Each monomial in $x$ and $y$ that appears in expansion of the series
 $\sum_{k\geq n+1} \,(-1)^{k-1}(x+y+xy)^k/k$ has total degree $\geq n+1$. Therefore,
 the equality of formal series $\log(1+x+y+xy)=\log(1+x)+\log(1+y)$ shows that in the reduced  form of the
 polynomial $r_{n+1}(x,y):= V_n(x+y+xy)-V_n(x)-V_n(y)$ all the monomials
have total degree $\geq n+1$. It follows that
 \begin{align*}
    p_n(x+my+xy)
    &= \ell(n)m^nV_n(x/m +y + x/m\cdot y)\\
    &=\ell(n)m^nV_n(x/m)+\ell(n)m^nV_n(y)+\ell(n)m^n r_{n+1}(x/m,y)\\
    &=p_n(x)+m^n\cdot v_n(y)+s_{n+1}(x,y),
  \end{align*}
where $s_{n+1}(x,y):=\ell(n)m^nr_{n+1}(x/m,y)$ and $v_n(y):=\ell(n)V_n(y)$. Since both $p_n$ and $v_n$ have integer coefficients, so must have
$s_{n+1}$. \end{proof}

Let $X$ be a finite  CW complex of dimension $d$ with skeleton
decomposition
 \[\emptyset=X_{-1}\subset X_0\subset X_1\subset \cdots \subset X_{d}=X.\]
Consider the induced filtration of $ K^*(X)$
\[K^*_q(X)=\ker \big( K^*(X)\to  K^*(X_{q-1})\big)=\mathrm{image}\big(K^*(X,X_{q-1})\to K^*(X)\big).\]
One has
$\{0\}=K^*_{d+1}(X)\subset K_{d}^*(X)\subset \cdots\subset K^*_1(X)\subset K^*_0(X)=K^*(X)$
and  \[K_q^*(X)K^*_r(X)\subset K_{q+r}^*(X).\]
Since the map $K^0(X_{2i+1})\to K^0(X_{2i})$ is injective, it follows that
$K_{2i+1}^0(X)=K_{2i+2}^0(X)$.
We will use only the even components of this filtration corresponding to $K^0(X)$, namely
\begin{equation}\label{eq:filtration}
 \{0\}=K^0_{2\lfloor d/2 \rfloor+2}(X)\subset K^0_{2\lfloor d/2 \rfloor}(X)\subset  \cdots\subset K^0_2(X)\subset K^0_0(X)=K^0(X).
\end{equation}
 Since  $\widetilde K^0(X)=K^0_1(X)=K^0_2(X)$ we have
\begin{equation}\label{eqn:100}
\widetilde K^0(X)^{j}\subset K^0_{2j}(X).
\end{equation}
In particular,
 $\widetilde K^0(X)^{\lfloor d/2 \rfloor+1} \subset K^0_{d+1}(X)$ and hence $\widetilde K^0(X)^{\lfloor d/2 \rfloor+1}=\{0\}$.
\begin{definition}\label{def:mu-invariants}
 Let $X$ be a finite  CW complex of dimension $d$. For each $ n \geq 1$
 we define  the map
\[\mu_n:\Vect_{m+1}(X)\to \widetilde K^0(X)/K^0_{2n+2}(X),\]
by  $\mu_n(E)=\pi_n(p_n([\widetilde E]))$ where   $\pi_n:K^0(X)\to \widetilde K^0(X)/K^0_{2n+2}(X)$ is the natural quotient map.
For $n\geq \lfloor d/2 \rfloor$, $\mu_n(E)=p_n([\widetilde E])\in\widetilde K^0(X)$ since
$ K^0_{2n+2}(X)=\{0\}$.
\end{definition}

\begin{theorem}\label{thm:poly-obstruction2} Let $X$ be a  finite  CW complex
and let $E,F \in \Vect_{m+1}(X)$.  If $O_E$ and $O_F$ are  isomorphic as $C(X)$-algebras,  then  $\mu_n(E)-\mu_n(F)$ is divisible by $m^n$, for  $n\geq 1$.
\end{theorem}
\begin{proof} Set $a=[\widetilde E]$ and $b=[\widetilde F]$. If $O_E\cong O_F$, then
by Theorem~\ref{Thm:complete-obstruction_O_E} there is $h\in \widetilde K^0(X)$ such that
$a=b+mh+bh$. Then,  by Lemma~\ref{lemma:ABC} (ii)
 \[p_n(a)=p_n(b+mh+bh)=p_n(b)+m^n v_n(h)+s_{n+1}(b,h),\]
and $s_{n+1}(b,h)\in \widetilde K^0(X)^{n+1}\subset K_{2n+2}^0(X)$ by \eqref{eqn:100}.
Thus $\mu_n(E)-\mu_n(F)=m^n\pi_n(v_n(h))$.
 \end{proof}
 As a corollary, we derive  Theorem~\ref{thm:k-theory-inv}, restated here as follows:
\begin{corollary}\label{cor:k-theory-inv}
 Let $X$ be a  finite  CW complex of dimension $d$ and let $E,F \in \Vect_{m+1}(X)$.
If $O_E\cong O_F$  as $C(X)$-algebras,  then
$p_{\lfloor d/2 \rfloor}([\widetilde E])-p_{\lfloor d/2 \rfloor}([\widetilde E])$ is divisible by $m^{\lfloor d/2 \rfloor}$ in $\widetilde K^0(X)$.
\end{corollary}
\begin{proof}
 If $n\geq \lfloor d/2 \rfloor$, then  $K_{2n+2}^0(X)=\{0\}$ and so
$p_n(a)-p_n(b)\in m^n \widetilde K^0(X)$.
\end{proof}
\begin{remark}\label{remark1}
 Let us note that
$\mu_{\lfloor d/2 \rfloor}(E)$  determines
$\mu_{\lfloor d/2 \rfloor+k}(E)$ for $k\geq 1$. Indeed, letting $n=\lfloor d/2 \rfloor$ it follows from \eqref{eqn:poly_n} that
\[\label{eqn:101}
 \mu_{n+k}(E)=\frac{\ell(n+k)}{\ell(n)}m^k \mu_{n}(E),
\]
since $\widetilde K^0(X)^{n+k}=\{0\}$.
 Let us  note that $\mu_{\lfloor d/2 \rfloor}(E)$  is also related to the lower order invariants. Indeed, it follows immediately from \eqref{eqn:poly_n} and \eqref{eqn:100}
that if $1\leq j\leq \lfloor d/2 \rfloor$, then
\[\frac{\ell(j)}{\ell(j-1)}\,m\,\mu_{j-1}(E)=\pi_{j,j-1}(\mu_{j}(E)),\]
where $\pi_{j,j-1}$ stands for the quotient map  $\widetilde K^0(X)/K^0_{2j+2}(X)\to \widetilde K^0(X)/K^0_{2j}(X)$. From this,
with $n=\lfloor d/2 \rfloor$, we obtain
\[\frac{\ell(n)}{\ell(n-j)}\,m^j\,\mu_{n-j}(E)=\pi_{n-j}(\mu_{n}(E)).
\]
Assuming that $\Tor(K^0_{2j}(X)/K^0_{2j+2}(X),\ZZZ/m)=0$ for all $j\geq 1$, and that $m$ and $\lfloor d/2 \rfloor!$ are relatively prime, it follows that if
 $\mu_{\lfloor d/2 \rfloor}(E)-\mu_{\lfloor d/2 \rfloor}(F)$ is divisible by $m^{\lfloor d/2 \rfloor}$, then $\mu_{j}(E)-\mu_{j}(F)$ is divisible by $m^{j}$ for all $j\geq 1$.

The groups $\widetilde K^0(X)/K^0_{2j}(X)$ are homotopy invariants of $X$,
and they are actually independent of the CW structure \cite{AH}.
Let $k^j(X)$ denote the reduced connective K-theory of $X$ and let $\beta:k^{j+2}(X)\to k^j(X)$ be the
Bott operation. One can identify $k^{2j+2}(X)$ with $K^0(X,X_{2j})$ in such a way that $\beta$
corresponds to the map $K^0(X,X_{2j})\to K^0(X,X_{2j-2})$.
Thus the image of $\beta^{j+1}:k^{2j+2}(X)\to k^0(X)\cong \widetilde K^0(X)$ coincides with
$K^0_{2j}(X)$, and hence $\mu_j(E)$ can be viewed as an element of $k^0(X)/\beta^{j+1}k^{2j+2}(X)$.
\end{remark}

\section{Cohomology invariants of $O_E$}\label{section:2+}
Let us recall that $V_n(x)=\log(1+x)_{[n]}$ and consider the polynomials
\[W_n(x)= \frac{n!}{\ell(n)}p_n(x)=n!m^n\,\log\left(1+\frac{x}{m}\right)_{[n]}=n!m^nV_n(\frac{x}{m})=\sum_{r=1}^n (-1)^{r-1} m^{n-r}\,\frac{n!}{r}x^r.\]

For a polynomial $P$ in variables $x_1$,...$x_n$, we assign to the variable $x_k$ the weight $k$ and
denote by $P(x_1,...,x_n)_{\langle n \rangle}$ the sum of all monomials of $P$ of total weight $n$.
For example if $P(x_1,x_2,x_3)=(x_1+\frac{x_2}{2}+\frac{x_3}{3})^2$, then $P(x_1,x_2,x_3)_{\langle 3 \rangle}=x_1x_2.$
Consider the polynomials
 \begin{align*}q_n(x_1,...,x_n)&=W_n\left(\frac{x_1}{1!}+...+\frac{x_n}{n!}\right)_{\langle n \rangle}=\sum_{r=1}^n(-1)^{r-1}m^{n-r}\frac{n!}{r}\left(\frac{x_1}{1!}+...+\frac{x_n}{n!}\right)^r_{\langle n \rangle}\\&=\sum_{r=1}^n(-1)^{r-1}m^{n-r}\sum_{ \substack{k_1+...+k_n=r,\\ k_1+2k_2+...+nk_n=n}}\frac{n!\,r!}{1!^{k_1}\cdots{n!}^{k_n}\,k_1!\cdots k_n!\,r}\,x_1^{k_1}\cdots x_n^{k_n}.\end{align*}
 Thus
 \begin{equation*}\label{qnn}
 q_n(x_1,...,x_n)=\sum_{ k_1+2k_2+...+nk_n=n}(-1)^{k_1+\cdots+k_n-1} m^{n-(k_1+\cdots+k_n)}\frac{n!\,(k_1+\cdots+k_n-1)!}{1!^{k_1}\cdots{n!}^{k_n}\,k_1!\cdots k_n!}\,x_1^{k_1}\cdots x_n^{k_n}.
 \end{equation*}
 Consider also the polynomials $r_n$  obtained from $q_n$ by taking $m=1$, i.e.
 \[r_n(x_1,...,x_n)=n!V_n\left(\frac{x_1}{1!}+...+\frac{x_n}{n!}\right)_{\langle n \rangle}=\sum_{r=1}^n(-1)^{r-1}\frac{n!}{r}\left(\frac{x_1}{1!}+...+\frac{x_n}{n!}\right)^r_{\langle n \rangle}.\]

\begin{lemma}\label{lemma:integralcoeff}
  The polynomials $q_n(x_1,...,x_n)$ and $r_n(x_1,...,x_n)$ have integer coefficients.
 \end{lemma}
\begin{proof} We have a factorization
 \[
\frac{n!\,(k_1+\dots+k_n-1)!}{1!^{k_1}\cdots{n!}^{k_n}\,k_1!\cdots k_n!}
=\frac{(k_1+2k_2+\dots+nk_n)!}{(k_1)!(2k_2)!\cdots (nk_n)!}\, a(1,k_1)\cdots a(n,k_n)(k_1+\dots+k_n-1)!,
 \]
 where $a(j,k)=\frac{(jk)!}{(j!)^k\,k!}$. It follows that the coefficient of $x_1^{k_1}\cdots x_n^{k_n}$  is an integer since it involves a multinomial coefficient and numbers $a(j,k)$ which are easily seen to be integers by using the recurrence formula $a(j,k)=\binom{jk-1}{j-1}a(j,k-1)$ where $a(j,1)=1$.
\end{proof}

Let us recall from \cite{Kar:k-theory} that the components of the Chern character $ch(E)=\sum_{k\geq 0}\, s_k(E)/k!$
involve integral stable characteristic classes $s_k$, also denoted by $ch_k$.
The classes $s_k$ have
 two important properties. If one sets $s_0(E)=\mathrm{rank}(E)$,
 then for $k\geq 0$:
\[s_k(E\oplus F)=s_k(E)+s_k(F)\]
\begin{equation}\label{s-classes}
s_k(E\otimes F)=\sum_{i+j=k}\frac{k!}{i!j!}s_i(E)s_j(F).
\end{equation}
We are now ready to prove Theorem~\ref{thm:cohomology-inv},  restated here for the convenience of the reader:
 \begin{theorem}\label{thm:poly-obstruction3}
   Let $X$ be a  compact metrizable space  and let $E,F \in \Vect_{m+1}(X)$.
If $O_E\cong O_F$  as $C(X)$-algebras,  then
$q_n(s_1(E),...,s_n(E))-q_n(s_1(F),...,s_n(F))$ is divisible by $m^{n}$ in $H^{2n}(X,\ZZZ)$, for each $n\geq 1$.
 \end{theorem}
\begin{proof}
Multiplying by $n!/\ell(n)$ in  Lemma~\ref{lemma:ABC}, (ii), we obtain
 \begin{equation}\label{eqn:W}
  W_n(y+mh+yh)-W_n(y)={m^n}{ n!}\,V_n(h)+\frac{n!}{\ell(n)}s_{n+1}(y,h).
 \end{equation}
Recall that $s_{n+1}$ is a polynomial with all monomials of degree at least $n+1$.
Let us make in \eqref{eqn:W} the substitutions
\[y=\frac{y_1}{1!}+\dots+\frac{y_n}{n!}, \quad h=\frac{h_1}{1!}+\dots+\frac{h_n}{n!},\]
where $y_k$ and $h_k$ have weight $k$. With these substitutions we have
\[W_n(y)_{\langle n \rangle}=q_n(y_1,...,y_n),\quad n!\,V_n(h)_{\langle n \rangle}=r_n(h_1,...,h_n),\]
whereas $\frac{n!}{\ell(n)}s_{n+1}(y,h)_{\langle n \rangle}=0$.
By  grouping the terms of $y+mh+yh$ of the same weight, we have
\[y+mh+yh=\frac{y_1+mh_1}{1!}+\frac{y_2+mh_2+2y_1h_1}{2!}+\cdots+\frac{y_n+mh_n+\sum_{i+j=n}\frac{n!}{i!j!}y_ih_{j}}{n!},\]
and hence
\[W_n(y+mh+yh)_{\langle n \rangle}=q_n(y_1+mh_1,\,y_2+mh_2+2y_1h_1\,,...\,,y_n+mh_n+\sum_{i+j=n}\frac{n!}{i!j!}y_ih_{j})\]
Thus, the equation
$W_n(y+mh+yh)_{\langle n \rangle}-W_n(y)_{\langle n \rangle}= n!\,V_n(h)_{\langle n \rangle}$
 implies that
\[q_n(y_1+mh_1,\,y_2+mh_2+2y_1h_1,\,...\,,y_n+mh_n+\sum_{i+j=n}\frac{n!}{i!j!}y_ih_{j})-q_n(y_1,...,y_n)\]
is equal to $m^nr_n(h_1,...,h_n)$.

Suppose now $O_E\cong O_F$. Then, by Theorem~\ref{Thm:complete-obstruction_O_E}, $[\widetilde E]=[\widetilde F]+m H+[\widetilde F]H$ for some
$H\in \widetilde K^0(X)$. Using ~\eqref{s-classes} it follows that for $k\geq 1$
\[s_k(E)=s_k(F)+m\,s_k(H)+\sum_{i+j=k}\frac{k!}{i!j!}s_i(F)s_j(H),\]
and hence
\[q_n(s_1(E),...,s_n(E))-q_n(s_1(F),...,s_n(F))=m^n\,r_n(s_1(H),...,s_n(H)).\]
\end{proof}
It follows by Theorem~\ref{thm:poly-obstruction3}, that the image of $q_n(s_1(E),...,s_n(E))$ in $H^{2n}(X,\ZZZ/m^n)$ is an invariant of $O_E$ for each $n\geq 1$.
The first four invariants in this sequence are:
\begin{itemize}
\item[] $\dot s_1(E)\in H^2(X;\mathbb{Z}/m)$
\item[] $m \dot s_2(E)-\dot s_1(E)^2\in H^4(X;\mathbb{Z}/m^2)$
\item[] $m^2 \dot s_3(E)-3m \dot s_1(E)\dot s_2(E)+2 \dot s_1(E)^3 \in H^6(X;\mathbb{Z}/m^3)$
\item[] $m^3\dot s_4(E) -m^2(3\dot s_2(E)^2 +4\dot s_1(E)\dot s_3(E)) +12m\dot s_1(E)^2 \dot s_2(E) -6\dot s_1(E)^4
\in H^8(X;\mathbb{Z}/m^4)$
\end{itemize}

The classes $s_k(E)$ are related to the Chern classes $c_k(E)$ via the
Newton polynomials:
\[s_k(E)=Q_k(c_1(E),...,c_k(E))\in H^{2k}(X;\ZZZ),\]
which express the symmetric power sum functions in terms of elementary symmetric functions $\sigma_i$. The first four Newton polynomials are
\begin{itemize}
 \item []$Q_1(\sigma_1)=\sigma_1$,
 \item []$Q_2(\sigma_1,\sigma_2)=\sigma_1^2-2 \sigma_2,$
 \item []$Q_3(\sigma_1,\sigma_2,\sigma_3)=\sigma_1^3-3\sigma_1\sigma_2+3\sigma_3$,
 \item []$Q_1(\sigma_1,\sigma_2,\sigma_3,\sigma_4)=
\sigma_1^4-4\sigma_1^2\sigma_2 +4\sigma_1\sigma_3+2\sigma_2^2-4\sigma_4.$
\end{itemize}

By expressing  $s_k$
in terms of Chern classes we obtain a sequence of
 characteristic classes of $E$ which are invariants of  $O_E$.
 The first four classes in the sequence are:
\begin{itemize}
\item[(1)] $\dot c_1(E)\in H^2(X;\mathbb{Z}/m)$
\item[(2)] $(m-1)\dot c_1(E)^2-2m \dot c_2(E)\in H^4(X;\mathbb{Z}/m^2)$
\item[(3)] $(m^2-3m+2) \dot c_1(E)^3-(3m^2-6m)\dot c_1(E)c_2(E)+3m^2 \dot c_3(E)\in H^6(X;\mathbb{Z}/m^3)$
\item[(4)] $(m^3-7m^2+12m-6)\dot c_1(E)^4-(4m^3-24m^2+24m) \dot c_1(E)^2 \dot c_2(E)+\\
(4m^3-12m^2)\dot c_1(E)\dot c_3(E)+(2m^3-12m^2)\dot c_2(E)^2-4m^3\dot c_4(E)
\in H^8(X;\mathbb{Z}/m^4)$
\end{itemize}
Here we denote by $\dot c_k(E)$ the image of the Chern class $c_k(E)$ under the coefficient map $H^{2k}(X,\ZZZ)\to H^{2k}(X,\ZZZ/m^k)$.

\begin{corollary}\label{line-bundles}
Let $X$ be a  compact metrizable space and let $L$ and $L'$ be two line bundles over $X$.
If $O_{m+L}\cong O_{m+L'}$ as $C(X)$-algebras, then
$q_n(1,...,1)(c_1(L)^n-c_1(L')^n)$ is divisible by $m^n$ for all $n\geq 1$.
\end{corollary}
 \begin{proof} If $L$ is a line bundle, then $s_k(L)=c_1(L)^k$. Since all monomials of $q_n$ have
   weight $n$,  $q_n(s_1(L),...,s_n(L))=q_n(1,...,1)c_1(L)^n$.
 \end{proof}
 Let us note that there is a more direct way to derive cohomology invariants for $O_E$.
\begin{proposition}\label{cor:fffirst-obstruction_O_E}
Let $X$ be  a  compact metrizable space.
Then $s_n(p_n([\widetilde E]))$ is an element of $H^{2n}(X,\ZZZ)$ whose image in  $H^{2n}(X,\ZZZ/m^n)$ is an invariant of $O_E$,   $n\geq 1$.
\end{proposition}
 \begin{proof}
Suppose that $O_E\cong O_F$ as $C(X)$-algebras. Then, we saw in the proof of the
 Theorem~\ref{thm:poly-obstruction2} that
$p_n([\widetilde E])-p_n([\widetilde F])=m^n c +d $ for some $c\in \widetilde K^0(X)$ and
$d\in  \widetilde K^0(X)^{n+1}$.
  From the multiplicative properties of the $s_n$-classes \eqref{s-classes}, one deduces that $s_n$ vanishes on $\widetilde K^0(X)^{n+1}$. Therefore $s_n(p_n([\widetilde E]))-s_n(p_n([\widetilde F]))=m^ns_n(c)$.
\end{proof}
We note that this is not really a new invariant, since it is not hard to prove that
\[s_n(p_n([\widetilde E])=\ell(n)q_n(s_1(E),...,s_n(E)).\]
Theorem~\ref{thm:poly-obstruction3} shows that we can remove the factor $\ell(n)$ and hence obtain a finer
invariant.

\section{Proof of Theorem~\ref{thm:main_result-intro1}}\label{section:3}
 Recall that $\OOO_{m+1}(X)$ denotes the set of isomorphism classes
of unital separable $C(X)$-algebras with all fibers isomorphic to $O_{m+1}$. These $C(X)$-algebras are automatically locally trivial if $X$ is finite dimensional.

For a discrete abelian group $G$ and $n\geq 1$ let $K(G,n)$ be an Eilenberg-MacLane space.
It is a connected CW complex $Y$ having just one nontrivial homotopy group $\pi_n(Y)\cong G$.
A $K(G,n)$ space is unique up to homotopy equivalence. For a CW complex $X$, there is an
isomorphism
$H^n(X,G)\cong [X,K(G,n)]$.

Let us recall from  \cite{Hatcher} that if $Y$ is a connected CW complex with $\pi_1(Y)$ acting trivially
on $\pi_n(Y)$ for $n\geq 1$, then $Y$ admits a Postnikov tower
\[\cdots\to Y_n\to Y_{n-1}\to \cdots \to Y_2\to Y_1=K(\pi_1(Y),1).\]
Each space $Y_n$ carries the homotopy groups of $Y$ up to level $n$.
More precisely, there exist compatible maps $Y\to Y_n$ that induce isomorphisms
$\pi_i(Y)\to \pi(Y_n)$ for $i\leq n$ and $\pi_i(Y_n)=0$ for $i>n$.
Each map $Y_n \to Y_{n-1}$
is a fibration with fiber $K(\pi_n(Y),n)$. Thus $Y_n$ can be thought as a twisted product
of $Y_{n-1}$ by $K(\pi_n(Y),n)$.
The space $Y$ is weakly homotopy equivalent to the projective limit
$\varprojlim Y_n$.

\begin{proposition}\label{Postnikov} Let $X$ be a finite  connected  CW complex and let $m\geq 1$ be an integer.
Then $|\OOO_{m+1}(X)|\leq |\widetilde H^{even}(X,\ZZZ/m)|$.
\end{proposition}
\begin{proof}
Let $Y$ be a CW complex weakly homotopy equivalent to the classifying space $B\Aut(O_{m+1})$
of principal $\Aut(O_{m+1})$-bundles.
Then, there are bijections
\[\OOO_{m+1}(X)\cong[X,B\Aut(O_{m+1})]\cong [X,Y].\]
The homotopy groups of $\Aut(O_{m+1})$ were computed in \cite{Dadarlat:homotopy-aut}.
That calculation gives
  $\pi_{2k-1}(Y)=0$ and $\pi_{2k}(Y)=\ZZZ/m$, $k\geq 1$.
Consequently, the Postnikov tower of $Y$ reduces to its even terms
\[\cdots\to Y_{2k}\to Y_{2k-2}\to \cdots \to Y_2=K(\ZZZ/m,2).\]
The   homotopy sequence of the fibration  $K(\ZZZ/m,2k)\to  Y_{2k}\to Y_{2k-2}$ gives
for all choices of the base points an exact sequence of sets
\[ [X,K(\ZZZ/m,2k)]\to [X,Y_{2k}]\to [X,Y_{2k-2}]. \]
This shows that
\[|[X,Y_{2k}]|\leq |[X,Y_{2k-2}]|\cdot |H^{2k}(X,\ZZZ/m)|.\]
 By Whitehead's theorem,  if $n> \dim(X)/2$, then the map $Y\to Y_{2n}$ induces a bijection $[X,Y]\cong [X,Y_{2n}]$.
It follows that \[|[X,Y]|\leq \prod_{1\leq k\leq n} |H^{2k}(X,\ZZZ/m)|=|\widetilde H^{even}(X,\ZZZ/m)|.\qedhere\]
\end{proof}
We consider  a commutative ring $R$ which admits a filtration by ideals
$$\cdots \subset R_{k+1}\subset R_k\subset \cdots \subset R_1=R$$ with the property that $R_q R_k\subset R_{q+k}$ and there is $n$
such that $R_{n+1}=\{0\}$.
On $R$ we consider the following equivalence relation: $a\sim b$ if there is $h\in R$ such that $a=b+mh+bh$. Let us denote by
$R/{\sim}$ the set of equivalence classes.

\begin{lemma}\label{countK0} Let $R$ be a filtered commutative  ring with $R_{n+1}=\{0\}$. Suppose that $\Tor(R_k/R_{k+1},\ZZZ/m)=0$
for all $k\geq 1$.
Then
$|R/{\sim}|=|R\otimes \ZZZ/m|=\prod_{k\geq 1}|R_k/R_{k+1}\otimes \ZZZ/m|$.
\end{lemma}
\begin{proof} Using the exact sequence for $\Tor$, we observe first that $\Tor(R_1/R_{k},\ZZZ/m)=0$
for all $k\geq 1$ and hence if $h\in R$ satisfies $mh\in R_k$ for
some $k$, then $h\in R_k$.
Using the exact sequences
$$0\to R_{k+1}\otimes \ZZZ/m\to R_k\otimes \ZZZ/m \to R_k/R_{k+1}\otimes \ZZZ/m\to 0$$
we see that $|R\otimes \ZZZ/m|=\prod_{k\geq 1}|R_k/R_{k+1}\otimes \ZZZ/m|$.
For each $k$ choose a finite subset $A_k\subset R_k$ such that the quotient
map $\pi_k:R_k\to R_k/R_{k+1}\otimes \ZZZ/m$ induces a bijective map
$\pi_k:A_k \to R_k/R_{k+1}\otimes \ZZZ/m$.
Consider the map $\eta:A_1\times A_2 \times \cdots \times A_n\to R$ defined by $\eta(a_1,\dots,a_n)=a_1+\cdots +a_n$.
To prove the proposition it suffices to show that $\eta$ induces a
bijection of $\bar{\eta}:A_1\times A_2 \times \cdots \times A_n\to R/{\sim}$.
First we verify that $\bar{\eta}$ is injective. Let $a_k,b_k\in A_k$, $1\leq k \leq n$
and assume that
\[a_1+\cdots+a_n\sim b_1+\cdots+b_n.\]
We must show that $a_k=b_k$ for all $k$.
Set $r_k=a_k+\cdots+a_n$ and $s_k=b_k+\cdots +b_n$. Then $r_k,s_k\in R_k$.
Since $a_1+r_2\sim b_1+ s_2$ there is $h_1\in R_1$
such that
\[a_1+r_2=b_1+ s_2+mh_1+b_1h_1+ s_2h_1\]
and hence $a_1-b_1-mh_1\in R_2$. Therefore $\pi_1(a_1)=\pi_1(b_1)$
and so $a_1=b_1$.
Arguing by induction, suppose that we have shown that $a_i=b_i$
for all $i\leq k-1$. Set $w=a_1+\cdots+a_{k-1}$.
By assumption $w+r_{k}\sim w+s_k$ and hence there
is $h\in R_1$ such that
\begin{equation}\label{rec}
w+r_k=w+s_k+mh+wh+s_kh.
\end{equation}
Let us notice that if $h\in R_i$ for some $i\leq k-1$, then
equation \eqref{rec} shows that $mh=(r_k-s_k)-wh-s_kh\in R_k\cup R_{i+1}=R_{i+1}$
and hence $h\in R_{i+1}$. This shows that in fact
$h=h_k\in R_k$. From equation \eqref{rec} we obtain
\[a_k-b_k-mh_k=s_{k+1}-r_{k+1}+(w+s_{k})h_k\in R_{k+1}.\]
This shows that $\pi_k(a_k)=\pi_k(b_k)$ and hence $a_k=b_k$.

It remains to verify that $\bar{\eta}$ is surjective.
In other words for any given $x_1\in R_1$ we must find $a_k\in A_k$, $1\leq k\leq n$,
such that $x_1\sim a_1+\cdots+a_n$. We do this by induction,
showing that for each $k\geq 1$ there exist $a_i\in A_i$, $1\leq i\leq k$ and  $x_{k+1}\in R_{k+1}$
such that
\[x_1\sim a_1+\cdots+a_k+x_{k+1},\]
and observe that $x_{n+1}=0$ since $R_{n+1}=\{0\}$.

By the definition of $A_1$ there is $a_1\in A_1$ such that
$\pi_1(a_1)=\pi_1(x_1)\in R_1/R_2\otimes \ZZZ/m$. Therefore there exist $h_1\in R_1$ and $y_2\in R_2$ such that $a_1=x_1+mh_1+y_2$.
Setting $x_2=x_1h_1-y_2\in R_2$ we obtain
\[x_1\sim x_1+mh_1+x_1h_1=a_1+x_2.\]
Suppose now that we found $a_i\in A_i$, $1\leq i\leq k-1$ and $x_k\in R_k$ such that
\[x_1\sim a_1+\cdots+a_{k-1}+x_{k}.\]
Let $a_k\in A_k$ be such that $\pi_k(a_k)=\pi_k(x_k)$. Then there exist
$h_k\in R_k$ and $y_{k+1}\in R_{k+1}$ such that
$a_k=x_k+mh_k+y_{k+1}$. Thus
\begin{align*}
    a_1+\cdots+a_{k-1}+x_{k}
    &\sim a_1+\cdots+a_{k-1}+x_{k}+mh_k+(a_1+\cdots+a_{k-1}+x_{k})h_k\\
    &=a_1+\cdots+a_{k-1}+a_{k}+x_{k+1}
  \end{align*}
  where $x_{k+1}=( a_1+\cdots+a_{k-1}+x_{k})h_k-y_{k+1}\in R_{k+1}$.
\end{proof}

\begin{theorem}\label{thm:main_result}
Let $X$ be a finite  connected CW complex of dimension $d$ and let $m\geq 1$ be an integer.  Suppose that $\Tor(H^*(X,\ZZZ),\ZZZ/m)=0$
and that $m\geq \lceil (d-3)/2\rceil$. Then:
\begin{itemize}
 \item[(i)] Any separable unital $C(X)$-algebra with fiber $O_{m+1}$ is isomorphic to $O_E$ for some $E\in Vect_{m+1}(X)$.
 \item[(ii)] If $E,F\in Vect_{m+1}(X)$, then $O_E\cong O_F$ as $C(X)$-algebras if and only if   there is $h\in  K^0(X)$  such that
$1-[E]=(1-[F])h$.
 \item[(iii)] The cardinality of $\OOO_{m+1}(X)$ is equal to  $|\widetilde K^0(X)\otimes \ZZZ/m|=|\widetilde H^{even}(X,\ZZZ/m)|$.
\end{itemize}
\end{theorem}
\begin{proof} Part (ii) is already contained in Theorem~\ref{Thm:complete-obstruction_O_E} but
we state it again nevertheless since a new proof of the implication  $(1-[E])\widetilde K^0(X)=(1-[F])\widetilde K^0(X)$ $\Rightarrow$
$O_E\cong O_F$
 is given here under the assumptions from the statement of the theorem.
   Recall that we defined an equivalence relation on $\widetilde K^0(X)$ by $a\sim b$ if and only
if $a=b+mh+bh$ for some $h\in \widetilde K^0(X)$.
 Let $\gamma: Vect_{m+1}(X)\to \widetilde K^0(X)/{\sim}$ be the map which takes $E$
 to the equivalence class of $[\widetilde E]=[E]-m-1$. We saw in Section~\ref{section:2} that
$1-[E]=(1-[F])k$ for some $k\in K^0(X)$  if and only if
$[\widetilde E]\sim [\widetilde F]$ in $\widetilde K^0(X)$, i.e. $\gamma(E)=\gamma(F)$.
Let $\omega: Vect_{m+1}(X) \to \OOO_{m+1}(X)$ the map which takes $E$ to the isomorphism class of the $C(X)$-algebra $O_E$.
We  shall construct a bijective map $\chi$ such that the diagram
\begin{equation*}
\xymatrix{
{Vect_{m+1}(X)}\ar[r]^{\omega}\ar@{->>}[d]_{\gamma}&
{\OOO_{m+1}(X)}\ar@{-->}[dl]^{\chi}\\
{\widetilde K^0(X)/{\sim} }
}\end{equation*}
is commutative.  Let us note that
in order to prove the parts (i), (ii) and (iii) of the theorem it suffices to verify the following four conditions.
\begin{itemize}
\item[(a)] $\gamma$ is surjective;
 \item[(b)] If $\omega(E)=\omega(F)$ then $\gamma(E)=\gamma(F)$;
  \item[(c)] $|\OOO_{m+1}(X)|\leq |\widetilde H^{even}(X,\ZZZ/m)|$;
 \item[(d)] $|\widetilde K^0(X)/{\sim}|=
|\widetilde K^0(X)\otimes \ZZZ/m|=|\widetilde H^{even}(X,\ZZZ/m)|$.
 \end{itemize}
Indeed,  from (a) and (b) we see that there is a well-defined  surjective map $\chi:\mathrm{image}(\omega)\to \widetilde K^0(X)/{\sim} $, given by $\chi(\omega(E))=
\gamma(E)$ and hence $|\widetilde K^0(X)/{\sim}|\leq |\mathrm{image}(\omega)|\leq |\OOO_{m+1}(X)|$. On the other hand
from   (c) and (d) we deduce that
 $|\OOO_{m+1}(X)|\leq |\widetilde K^0(X)/{\sim}|$. Altogether this implies that
 $\omega$ is surjective and $\chi$ is bijective.

It remains to verify the four conditions from above.
If $m\geq  \lceil (d-3)/2 \rceil$,
   then the map $Vect_{m+1}(X)\to \widetilde K^0(X)$, $E\mapsto [E]-m-1$ is surjective
   by \cite[Thm.~1.2]{Hus:fibre} and  this implies  (a).
Condition (b) follows from Theorem~\ref{Thm:complete-obstruction_O_E} and condition (c) was proved in
Proposition~\ref{Postnikov}.
It remains to verify condition (d) using the assumption that
    $\Tor(H^*(X,\ZZZ),\ZZZ/m)=0$. The first step is to use a known argument to deduce the absence of $m$-torsion in the   K-theory of $X$ and its skeleton filtration. We will then appeal to Lemma~\ref{countK0} to conclude the proof.

Let $p$ be a prime which divides $m$. Then $\Tor(H^*(X,\ZZZ),\ZZZ/p)=0$ by assumption.
Let $\ZZZ_{(p)}$ denote $\ZZZ$ localized at $p$, i.e. the subring of $\QQQ$ consisting of all fractions with denominator  prime to $p$.
Let $(E_r,d_r)$ be the Atiyah-Hirzebruch spectral sequence  $H^*(X,\ZZZ)\Rightarrow K^*(X)$.
Recall that $E_2^{s,t}=H^{s}(X,K^t(pt))$ and  $E_\infty^{s,t}=K^{s+t}_s(X)/K^{s+t}_{s+1}(X)$, see \cite{AH}.
Since $\ZZZ_{(p)}$ is torsion free, it follows from the universal coefficient theorem
that the  spectral sequence
$(E_r\otimes \ZZZ_{(p)}, d_r\otimes 1)$ is convergent to $K^*(X)\otimes \ZZZ_{(p)}$.
On the other hand since all the differentials $d_r$ are torsion operators  by \cite[2.4]{AH}
and
since $H^*(X,\ZZZ)$ has no p-torsion, it follows that $d_r\otimes 1=0$ for all $r\geq 2$
and hence $E_2\otimes\ZZZ_{(p)}=E_\infty\otimes\ZZZ_{(p)}$.
Therefore for all $q\geq 0$
\begin{equation}\label{eqn:torsion_control}
 H^{2q}(X,\ZZZ)\otimes \ZZZ_{(p)}\cong \big(K^0_{2q}(X)/K^0_{2q+2}(X)\big)\otimes\ZZZ_{(p)}.
\end{equation}
Since $\Tor(G\otimes\ZZZ_{(p)}, \ZZZ/p)\cong \Tor(G, \ZZZ/p)$ for all finitely generated
abelian groups $G$, it follows that  $\Tor(K^0_{2q}(X)/K^0_{2q+2}(X),\ZZZ/p)=0$ for any
prime $p$ that divides $m$. Therefore for all $q\geq 0$ we have
\[\Tor(K^0_{2q}(X)/K^0_{2q+2}(X),\ZZZ/m)=0.\]
This enables us to apply
Lemma~\ref{countK0}  for the ring $R=\widetilde K^0(X)= K_2^0(X)$ filtered  by the ideals
$R_q=\widetilde K^0_{2q}(X)$ to obtain that
\[|\widetilde K^0(X)/{\sim}|=|\widetilde K^0(X)\otimes \ZZZ/m|=\prod_{q\geq 1}|K^0_{2q}(X)/K^0_{2q+2}(X)\otimes \ZZZ/m|.\]
Since $\Tor(H^*(X,\ZZZ),\ZZZ/m)=0$, we have $H^{2q}(X,\ZZZ/m)\cong H^{2q}(X,\ZZZ)\otimes \ZZZ/m$.
From equation \eqref{eqn:torsion_control} we deduce that
\[H^{2q}(X,\ZZZ)\otimes \ZZZ/m\cong \big(K^0_{2q}(X)/K^0_{2q+2}(X)\big)\otimes\ZZZ/m.\]
This completes the proof (d).
\end{proof}
\section{Proof of Theorem~\ref{thm:main_result2-intro}}\label{section:2algebra}
\begin{lemma}\label{lemma:C} Let $R$ be a filtered commutative  ring with $R_{n+1}=\{0\}$.
Suppose that  $\Tor(R_k/R_{k+1},\ZZZ/m)=0$ for all $k\geq 1$ and that $m$ and $n!$ are relatively prime.  If $a,b\in R$  and $p_n(a)-p_n(b)\in m^nR$, then $a\sim b$.
\end{lemma}
\begin{proof}
 We prove this by induction on $n$.  Suppose first that $n=1$. Then $p_1(a)-p_1(b)=a-b\in mR$ by assumption and so
 $a=b+mh$ for some $h\in R$. Since $R_2=\{0\}$ by assumption, $bh=0$ and so $a=b+mh+bh$, i.e. $a\sim b$.
 Suppose now that the statement is true for a given $n$ for all filtered rings $R$ as in the statement.
 Let $R$ be now a filtered ring such that $R_{n+2}=\{0\}$, $m$ and $(n+1)!$ are relatively prime and $\Tor(R_k/R_{k+1},\ZZZ/m)=0$ for all $k\geq 1$. Consider the  ring $S:=R/R_{n+1}$ with filtration
 $S_k=R_k/R_{n+1}$, $S_{n+1}=\{0\}$, and the quotient map $\pi:R\to S$.
 Let $a,b \in R$ satisfy $p_{n+1}(a)-p_{n+1}(b)\in m^{n+1}R$.
Since \[p_{n+1}(x)=\frac{\ell(n+1)}{\ell(n)} m p_{n}(x)+(-1)^n \frac{\ell(n+1)}{n+1}x^{n+1}.\]
 we deduce that
$\frac{\ell(n+1)}{\ell(n)} m\big(p_{n}(\pi(a))-p_{n}(\pi(b))\big)\in m^{n+1}S.$
Since $Tor(S,\ZZZ/m)=0$ and $(n+1)!$ and $m$ are relatively prime it follows that $p_{n}(\pi(a))-p_{n}(\pi(b))\in m^nS$. Since $S_{n+1}=0$ we obtain by the inductive hypothesis that $\pi(a)\sim \pi(b)$
in $S$ and hence $a=b+mh+bh+r_{n+1}$ for some $h\in R$ and $r_{n+1}\in R_{n+1}$.
We have that $(b+mh+bh)\cdot r_{n+1}=0$ and $r_{n+1}^2=0$ as these are elements of $R^{n+2}\subset R_{n+2}=\{0\}$.
Therefore, by  Lemma~\ref{lemma:ABC}(i),
\begin{align*}
p_{n+1}(a)=p_{n+1}(b+mh+bh+r_{n+1})&=p_{n+1}(b+mh+bh)+p_{n+1}(r_{n+1})\\
&=p_{n+1}(b+mh+bh)+\ell(n+1)m^nr_{n+1} \end{align*}
On the other hand,
$p_{n+1}(b+mh+bh)=p_{n+1}(b)+m^{n+1}v_{n+1}(h)$ by Lemma~\ref{lemma:ABC}(ii), since   $R_{n+2}=\{0\}$.
Therefore\[p_{n+1}(a)-p_{n+1}(b)=m^{n+1}v_{n+1}(h)+\ell(n+1)m^nr_{n+1}.\]
Since $p_{n+1}(a)-p_{n+1}(b)\in m^{n+1}R$ by assumption, we obtain that $\ell(n+1)m^nr_{n+1}\in m^{n+1}R$. Since $\Tor(R,\ZZZ/m)=0$ we deduce that $\ell(n+1)r_{n+1}\in mR$ and hence
that $r_{n+1}=m\, h_{n+1}$ for some $h_{n+1}\in R$ since $m$ is relatively prime to $\ell(n+1)$. We must have
that in fact $h_{n+1}\in R_{n+1}$ since $\Tor(R/R_{n+1},\ZZZ/m)=0$ and hence
$b\,h_{n+1}=0$. It follows that $a\sim b$ since we can now  rewrite $a=b+mh+bh+r_{n+1}$ as
$a=b+m(h+h_{n+1})+b(h+h_{n+1})$.
\end{proof}
We are now in position to prove Theorem~\ref{thm:main_result2-intro}.
\begin{proof}  By Theorem~\ref{Thm:complete-obstruction_O_E} it suffices to show that $[\widetilde E]\sim [\widetilde F]$ whenever $p_{\lfloor d/2 \rfloor}([\widetilde  E])-p_{\lfloor d/2 \rfloor}([\widetilde F])$ is divisible by $m^{\lfloor d/2 \rfloor}$. We have seen in the proof of Theorem~\ref{thm:main_result} that if $\Tor(H^*(X,\ZZZ),\ZZZ/m)=0$ then $\Tor(K^0_{2q}(X)/K^0_{2q+2}(X),\ZZZ/m)=0$.
Therefore the desired implication follows  from Lemma~\ref{lemma:C} applied to the ring $\widetilde K^0(X)$ filtered as in \eqref{eq:filtration} with $n=\lfloor d/2 \rfloor$.
 \end{proof}
\section{ suspensions}\label{section:susp}
In this final part of the paper we study $\OOO_{m+1}(SX)$ and the image of the map $\mathrm{Vect}(SX)\to\OOO_{m+1}(SX)$.
We shall use the universal coefficient exact sequence
\[
0\to K^1(X)\otimes \ZZZ/m\stackrel{\bar{\rho}}\to K^1(X,\ZZZ/m)\stackrel{\beta}\to \Tor(K^0(X),\ZZZ/m)\to 0,
\]
where $\beta$ is the Bockstein operation
and $\bar{\rho}$ is induced by the coefficient map $\rho$.
\begin{theorem}\label{thm:susp} Let $X$ be a compact metrizable space and let $m\geq 1$.
\begin{itemize}
\item[(i)] There is a bijection $\gamma: \OOO_{m+1}(SX)\to K_1(C(X)\otimes O_{m+1})\cong K^1(X,\ZZZ/m)$.
\item[(ii)] If $E,F\in Vect_{m+1}(SX)$, then $O_E\cong O_F$ as $C(SX)$-algebras if and only if
$[E]-[F]\in m K^0(SX)$.
\item[(iii)]
If $A\in  \OOO_{m+1}(SX)$ and $\beta(\gamma(A))\neq 0$, then $A$ is not isomorphic
to $O_E$ for any $E\in Vect_{m+1}(SX)$.
\end{itemize}
\end{theorem}
\begin{proof} Part (i) is proved in \cite{Dad:bundles-fdspaces}. We revisit the argument from \cite{Dad:bundles-fdspaces}
as it is needed for the proof of the other two parts. Let $v_1,...,v_{m+1}$ be the canonical generators of $O_{m+1}$.
There is natural a  map $\gamma_0:\Aut(O_{m+1})\to U(O_{m+1})$ which maps an automorphism $\varphi$ to the unitary
$\sum_{j=1}^{m+1} \varphi(v_j)v_j^*$. We showed in \cite[Thm.~7.4]{Dad:bundles-fdspaces} that
$\gamma_0$ induces a bijection of homotopy classes $[X,\Aut(O_{m+1})]\to [X,U(O_{m+1})]$.
By \cite{Ror:encyclopedia} there is a $*$-isomorphism $\nu:O_{m+1}\to M_{m+1}( O_{m+1})$.
We have bijections
$\OOO_{m+1}(SX)\cong [SX,B\Aut(O_{m+1})]\cong [X,\Aut(O_{m+1})]$ and
\[
\xymatrix{
[X,\Aut(O_{m+1})]\ar[r]^-{(\gamma_0)_*} &[X,U(O_{m+1})]\ar[r]^-{\nu_*} &[X,U(M_{m+1}(O_{m+1}))]
\cong K_1(C(X)\otimes \OOO_{m+1})}.\]
The composition of these maps defines the bijection $\gamma$ from (i).
We are now prepared to prove (ii) and (iii).
Consider the monomorphism of groups
$\alpha: U(m+1)\to\Aut (O_{m+1})$ introduced in \cite{EHW}. If $u\in U(m+1)$ has components $u_{ij}$, then
$\alpha_u(v_j)=\sum_{i=1}^{m+1} u_{ij}v_i$.
The map $\alpha$ induces a map $BU(m+1)\to B\Aut(O_{m+1})$ which in its turn induces
the natural map $\alpha_*:Vect_{m+1}(Y)\to \OOO_{m+1}(Y)$ that we are studying.
Let $\eta$ be the composition of the maps from the diagram
\[
\xymatrix{
U(m+1)\ar[r]^{\alpha}& \Aut(O_{m+1})\ar[r]^{\gamma_0}
&U(O_{m+1})\ar[r]^-{\nu}&U(M_{m+1}(O_{m+1})).
}\]
 An easy calculation shows that $\eta(u)=\sum_{i,j=1}^{m+1}u_{ij}\nu(v_iv_j^*)$, where $u_{ij}$ are the components of the unitary $u$.
 Let us observe that $\eta$ is induced by a unital $*$-homomorphism
$\bar{\eta}:M_{m+1}(\CCC)\to M_{m+1}(O_{m+1})$.  It follows that there is a unitary $w\in M_{m+1}(O_{m+1})$
such that $w\bar{\eta}(a)w^*=a\otimes 1_{O_{m+1}}$, for all $a\in M_{m+1}(\CCC)$.
This implies that $\eta$ will induce the coefficient map $\rho: K^1(X)\to K^1(X,\ZZZ/m)$.

We a commutative diagram
\[
\xymatrix{
Vect_{m+1}(SX)\ar[d]_{\alpha_*}\ar[r]& [SX,BU(m+1)]\ar[d]\ar[r]& [X,U(m+1)]\ar[d]_{\alpha_*}\ar@{=}[r] &[X,U(M_{m+1}(\CCC))]\ar[d]^{\eta_*}\\
\OOO_{m+1}(SX)\ar[r]& [SX,B\Aut(O_{m+1})]\ar[r]& [X,\Aut(O_{m+1})]\ar[r]^-{(\nu\gamma_0)_*} &[X,U(M_{m+1}(O_{m+1}))]
}\]
and hence a commutative diagram
\[
\xymatrix{
&Vect_{m+1}(SX)\ar[r]^{\alpha_*}\ar[d]& \OOO_{m+1}(SX)\ar[d]^{\gamma}\\
                   K^1(X)\ar[r]^{\times m}  &  K^1(X)\ar[r]^{\rho}               &  K^1(X,\ZZZ/m)\ar[r]^-\beta&\Tor(K^0(X),\ZZZ/m)\to 0
}\]
Now both $(ii)$ and $(iii)$ follow from the commutativity of the diagram above and the exactness of its bottom row.
\end{proof}

\end{document}